\documentclass[11pt,reqno]{amsart}
\usepackage{amsmath,amsthm,amsfonts,amssymb,mathrsfs,bm,graphicx,stmaryrd}
\usepackage[nobysame]{amsrefs}

\usepackage[usenames,dvipsnames]{color}
\usepackage{xcolor}
\usepackage[colorlinks=true,linkcolor=blue]{hyperref}

\definecolor{naomi}{rgb}{0.6, 0.3, 0.8}

\usepackage[letterpaper,hmargin=1.0in,vmargin=1.0in]{geometry}
\parskip 	\smallskipamount

\newcommand{\abbr}[1]{{\sc\lowercase{#1}}}
\numberwithin{equation}{section}

\newtheorem{theorem}{Theorem}[section]
\newtheorem{lemma}[theorem]{Lemma}

\newtheorem{proposition}[theorem]{Proposition}

\newtheorem{remark}[theorem]{Remark}

\def\vI{v_{\textsf{I}}}
\def\vR{v_{\textsf{R}}}
\def\N{\mathbb{N}}
\def\bI{\mathbf{I}}
\def\P{\mathbb{P}}
\def\Z{\mathbb{Z}}
\def\R{\mathbb{R}}
\def\S{\mathbb{S}}
\def\bS{\mathbf{S}}
\def\B{\mathbb{B}}
\def\Q{\mathbb{Q}}
\def\cS{\mathcal{S}}
\def\bSig{\boldsymbol \Sigma}
\def\bmu{\boldsymbol \mu}
\def\bro{{\bf \widehat{r}}}

\def\E{\mathbb{E}}

\def\Re{\mathrm{Re}}
\def\Im{\mathrm{Im}}
\def\C{\mathbb{C}}
\def\D{\mathbb{D}}
\def\sinc{\mathrm{sinc}}

\newcommand{\var}{\mathrm{Var}\,}
\newcommand{\cov}{\mathrm{Cov}\,}
\newcommand{\vep}{\varepsilon}

\newcommand{\cJ}{\mathcal{J}}
\newcommand{\calN}{\mathcal{N}}

\newcommand{\cs}{\mathcal{S}}

\newcommand{\p}{\varphi}
\newcommand{\lm}{\lambda}

\begin{document}

\title[Exponential concentration for {zeroes}]{Exponential Concentration for 
Zeroes\\ of Stationary Gaussian Processes}
\author{Riddhipratim Basu}
\author{Amir Dembo}
\author{Naomi Feldheim}
\author{Ofer Zeitouni}
\date{\today}
\thanks{
The research of RB is partially supported by a Simons Junior Faculty Fellowship. 
The research of NF is supported by an NSF postdoctoral fellowship. 
The research of AD is supported by NSF grant DMS-1613091, and 
both AD and OZ are supported by BSF grant 2014019.}


\maketitle

\begin{abstract}
We show that for any centered stationary Gaussian process 
of integrable covariance, whose spectral measure has compact support,
or finite exponential moments (and some additional regularity),
the number of zeroes of the process in $[0,T]$ 
is within $\eta T$ of its mean value,
up to an exponentially small in $T$ probability.
\end{abstract}

\section{Introduction}

The study of zeroes of Gaussian stationary processes goes back at least to 
Kac~\cite{Kac} and Rice~\cite{Rice}. Since then, 
much work on this topic appeared in the  
statistics, physics, mathematics and engineering literature.
One of the earliest and most fundamental results in this area is the Kac-Rice formula, which calculates the mean number of zeroes in any interval. A similar formula may be written for the variance of the number of zeroes, but it is much harder to analyze. It was only many years after Kac and Rice that fluctuations and central limit theorems were better understood, with works by Cuzick~\cite{Cuz}, Slud~\cite{Slud}, Aza\"is-Le\'on~\cite{AL} and others. 
Questions of large deviations, that is, estimation of the rare event of having many more or much less zeroes than expected in a long interval, remained almost unexplored. One particular such event is that of having no zeroes at all in a long interval, which is also known by the name of ``persistence''. Results (and speculations) about this event were initiated by Slepian \cite{Slep} and were better understood only recently \cite{FFN}.

In the meantime, complex zeroes of certain Gaussian analytic functions received much attention. Most notably, zeroes of the Fock-Bargmann model were introduced by Sodin-Tsirelson \cite{ST1} and extensively studied by many authors since then. This model has a remarkable property: its zeroes form a point process in the plane with quadratic repulsion, and invariance of distribution under all planar isometries. 
Sodin-Tsirelson proved the asymptotic normality of these zeroes in \cite{ST1}, and moreover, an exponential concentration of the zeroes around the mean in \cite{ST3}
(See also 
\cite{Manju,NSV}
for more about concentration, and 
\cite{NS_zeros}
for finer results on asymptotic normality).
Exponential concentration was proved for other related models, such as nodal lines of spherical harmonics in \cite{NS_nodal}. Inspired by these works, the question of concentration for real zeroes got some attention \cites{Sodin}, \cite[Thm. 2c1]{Tsirel}, but, until now, was not settled even for a single particular example.

The aim of this paper is to prove exponential concentration for \emph{real zeroes} of certain Gaussian stationary functions on $\R$, which have an analytic extension to a strip in the complex plane and smooth spectral density. These conditions allow us to use tools from complex analysis, thus generalizing the mechanics of the aforementioned works on the Fock-Bargmann model.

\bigskip
We consider here centered Gaussian stationary processes $\{X(t): t \in \R\}$ 
having a.s.\ \emph{absolutely continuous} sample path. 
That is, random absolutely continuous functions $f: \R \to \R$, 
whose finite marginal distributions are mean-zero multi-normal,
invariant to real shifts. Normalizing \abbr{wlog} the process $X$ 
to have variance one, its joint law is determined by
$r(t-s) := \cov[X(t),X(s)]$.
Here $r:\R\to\R$ is a continuous, positive semi-definite function (with $r(0)=1$). By Bochner's theorem, this yields
the existence of
a probability measure $\rho$ on $\R$,  
called the \emph{spectral measure},  
such that
\begin{equation}\label{eq:cov-spec}
r(t) 
= \int_{\R} e^{-i \lm t} \ d\rho(\lm) \,.
\end{equation} 
We further assume throughout that $\int \lm^2 d\rho < \infty$, or 
equivalently that $r(t)$ has finite second derivative at $t=0$
(in which case $r(t)$ is twice
continuously differentiable and $-r''(0)=\int \lm^2 d \rho$),
and 
let $N_{X}(I) = |\{t \in I : f(t) = 0 \}|$ count the number of 
zeroes, possibly infinite, of such a process in the 
interval $I \subset \R$. Since $X$ is stationary, 
$\E N_X([0,T]) = \alpha T$ for $\alpha := \E N_{X}([0,1])$ and 
from the Kac-Rice formula we have that in this case 
\begin{equation}\label{eq:Kac-Rice}
\alpha = 
\frac{1}{\pi} \Big( \int \lm^2 d\rho \Big)^{\frac{1}{2}} < \infty \,.
\end{equation}
Indeed, $\alpha/2$ is 
the expected number of $0$-upcrossings by $X([0,1])$, 
all of which are strict (see \cite[Theorem 7.3.2]{LR}),
and $X^{-1} \{0\}\cap [0,1]$ has expected size $\alpha$ since it a.s. 
consists of only the (strict) $0$-upcrossings and $0$-downcrossings
of $X([0,1])$ (see \cite[Theorem 7.2.5]{LR}).
Our objective is to prove 
the following exponential concentration of $N_X([0,T])$.
\begin{theorem}\label{t:nice}
Suppose the centered stationary Gaussian process $\{X(t):t\in \R\}$
of a compactly supported
spectral measure has an
integrable covariance (ie. $\int |r(t)| dt < \infty$).
Then, for some $C<\infty$ and $c(\cdot)>0$,
\begin{equation}\label{eq:exp-conc}
\P\Big(\,|N_{X}([0,T]) - \alpha T| \ge \eta T\,\Big) \leq C e^{-c(\eta) T} \,,
\qquad 
\forall \eta>0, \;\; T<\infty \,.
\end{equation}
Further, if the spectral measure has only a finite 
exponential moment, namely, for some $\kappa > 0$
\begin{equation}\label{eq:exp-mom-nd}
\int_{\R} e^{|\lm| \kappa} \, d \rho(\lm)  < \infty \,, 
\end{equation}
then \eqref{eq:exp-conc} holds whenever $\int |r(t;\kappa_o)| dt <\infty$ 
for some $\kappa_o \in (0,\kappa/2)$ and 
\begin{equation}\label{dfn:r-kappa-o}
r(t;\kappa_o) := \int_\R \cos(t \lm) \cosh (2 \kappa_o \lm) d \rho (\lm) \,.
\end{equation}
\end{theorem}
\begin{remark} Theorem \ref{t:nice} applies for example to
any spectral measure whose compactly supported density is 
in $W^{1,2}(\R)$, 
as well as 
covariances such as $r(t)=e^{-t^2/2}$ or $r(t)=1/(1+t^2)$
(of spectral densities $p(\lm)=\frac{1}{\sqrt{2\pi}} e^{-\lm^2/2}$,
and $p(\lm)=\frac{1}{2} e^{-|\lm|}$, respectively), for which 
\eqref{eq:exp-mom-nd} holds and the \abbr{lhs} of 
\eqref{dfn:r-kappa-o} is integrable.
\end{remark}
As shown next, the lower tail in \eqref{eq:exp-conc}
holds under much weaker regularity assumptions.
\begin{proposition}
\label{t:lower}
Suppose the centered stationary Gaussian process $\{X(t):t\in \R\}$
has an absolutely continuous sample path and bounded, continuous spectral 
density $p(\lm)$ with $\int \lm^2 p(\lm) d\lm < \infty$.
Then, for some
$C<\infty$ and $c(\cdot)>0$ we have:
\begin{equation}\label{eq:lower-tail}
\P\Big(\, N_{X}([0,T]) - \alpha T \leq - \eta T\,\Big)\leq Ce^{-c(\eta) T} \,,
\qquad 
\forall \eta>0, \;\; T<\infty \,.
\end{equation}
\end{proposition}
\noindent 
Proposition \ref{t:lower} is a consequence of our next result, on  
exponential concentration of the number of sign changes in $[0,T]$, 
for any \emph{discrete-time}
centered stationary Gaussian process $\{Y_k: k \in \Z\}$  
of continuous spectral density. 
\begin{theorem}
\label{t:sc}
Suppose $\{Y_k : k \in \Z\}$ is a centered stationary Gaussian process
whose spectral measure $\rho_Y$ has a continuous density
$p_Y(\lm)$ (supported within $[-\pi,\pi]$). Then, for 
\begin{equation}\label{def:sign-changes}
N^\star_Y (T) := \sum_{k=0}^{T-1} \mathbf{1}_{\{Y_kY_{k+1}<0\}}\,,
\end{equation}
some $C < \infty$ and $c(\cdot)>0$,
\begin{equation}\label{conc:sign-changes}
\P\Big(\,|N_Y^\star(T)-\E N_Y^\star(T)|\geq \eta T\,\Big)\leq Ce^{-c(\eta) T} \,,
\qquad 
\forall \eta>0, \;\; T \in \N \,. 
\end{equation}
\end{theorem}

\begin{remark}
In the setting of Theorem \ref{t:sc} one has  
that $T^{-1} \sum_{k=0}^{T-1} h(Y_k,Y_{k+1})$ satisfies the \abbr{LDP}
for any fixed $h \in C_b(\R^2)$,
with a convex, good rate function (see \cite[Theorem 4.25]{BJ}).
While this may be extended to $N^\star_Y(T)$ 
by suitable approximations, we do not follow this route 
since the rate function is in any case
not easily identifiable (see \cite[Sect. 7(a)]{BJ}).
\end{remark}

Proposition \ref{t:lower} will follow from Theorem \ref{t:sc}, using the key observation that having few zeroes of a continuous time process implies few sign changes of its restriction to a lattice.
This approach does not work for the 
more challenging
upper tail in \eqref{eq:exp-conc}, since having many zeroes does not imply having many sign-changes on a lattice,
and to establish the upper tail we require  
the following decay and regularity assumptions 
about the spectral density of $X$.

\noindent
\textbf{Assumption A:} 
The spectral measure is non-atomic and 
has finite exponential moment as in \eqref{eq:exp-mom-nd}.
Further, the covariance functions
 \begin{equation}\label{dfn:ry-ell}
r_{\ell}(x;y) := \int_{\R} e^{-i \lm x} \varphi^{\ell}(\lm y) d \rho (\lm) \,, 
\quad \ell =1,2 \,, \qquad \varphi(\lm) := \sinh(\lm)/\lm \,, 
\end{equation}
and their $x$-derivatives, 
satisfy for some $\kappa' \in (0,\kappa/2)$ and finite $x_\star$,
\begin{equation}
\omega_\star (k) : = 
\sum_{j \ge k} |r(j x_\star)| + 
\sup_{|y| < \kappa'} \Big\{ 
\sum_{j \ge k} |r_1'(j x_\star;2y)| + 
\sum_{j \ge k} |r''_{2}(j x_\star;y)| \Big\} \to 0  \,, 
\;\; \mbox{ when } \;\; k \to \infty \,.
\label{eq:reg-spect-dens}
\end{equation}
Equipped with Assumption A, we state our main (technical) result.
\begin{theorem}
\label{t:upper}
Subject to Assumption A we have 
for some $C<\infty$ and $c(\cdot)>0$, the exponential upper tail
\begin{equation}\label{eq:upper-tail}
\P\Big(\, N_{X}([0,T])- \alpha T \geq 3 \eta T\Big) \leq Ce^{-c (\eta) T}\,,
\qquad 
\forall \eta>0, \;\; T < \infty \,. 
\end{equation}
\end{theorem}
\noindent
In particular, we recover Theorem \ref{t:nice} from
Proposition \ref{t:lower} and Theorem \ref{t:upper}, thanks to 
the following explicit 
sufficient condition for Assumption A.
\begin{proposition}\label{p:suff-A}  
Assumption A is satisfied, the spectral measure 
$\rho(\cdot)$ has a bounded, continuous density, and 
a.s. the sample path $t \mapsto X(t) \in C^\infty(\R)$, 
when either of the following holds:
\newline
(a). The support of $\rho(\cdot)$ is compact and
$\int |r(t)| dt < \infty$ for the covariance $r(t)$ of \eqref{eq:cov-spec}.
\newline
(b). Condition \eqref{eq:exp-mom-nd} holds and 
$\int |r(t;\kappa_o)| dt <\infty$ for the covariance 
$r(t;\kappa_o)$ of \eqref{dfn:r-kappa-o}.
\end{proposition}

\medskip
It is reasonable 
when seeking the exponential concentration of $N_X([0,T])$,
to require 
smoothness of the covariance $r(\cdot)$,
such as having all spectral moments 
finite (or the stronger condition \eqref{eq:exp-mom-nd}). 
Indeed, such exponential concentration implies 
the finiteness of all moments $m_k := \E [ N_X([0,T])^k ]$, 
with $m_k = O((T \vee k)^k)$, and such conditions appear 
in previous studies concerning the finiteness of $\{m_k\}$. 
For instance, Nualart and Wschebor \cite{NW} show that $m_k<\infty$ for all $k$ 
when $t \mapsto r(t)$ is real-analytic 
(hence all spectral moments are finite), while when
$\int \lm^4 d\rho = \infty$, Cuzick \cite{Cuz2} can prove only
the finiteness of $m_k$ up to a certain order $k_o$. Similarly, 
Longuett-Higgins \cite{LH} shows that for 
$r(t)$ real-analytic,
$q_k(\tau):=\P(N_X(0,\tau) \ge k)$ decays, for $\tau \to 0$,
as $c(k) \tau^{\frac 1 2 k^2 + O(k)}$ 
(indicative of the mutual repulsion of zeroes), 
while 
with a discontinuity of $r^{(3)}$ at the origin
the decay of $q_k(\tau)$
is merely $c(k)\tau^2$ for all $k$ 
(so having a pair of nearby zeroes, the probability 
of $k$ extra zeroes within the same short interval is $O_k(1)$).

\medskip
A natural path towards proving the upper tail in our concentration result is to 
improve Cuzick's results on moments $m_k$ \cite{Cuz} or Longuett-Higgins estimates on the tail of the number of zeroes $q_k$ \cite{LH}, so as to get accurate asymptotics of those quantities in $k$. Efforts in this direction were 
made by many authors (e.g. 
\cite{AL} and the references within).
However, in our context it requires a lower bound on the determinant of
nearly singular matrices (specifically, the covariance matrices for values 
of $X(t)$ at a short range), at a level of accuracy which seems out of reach.
We bypass this difficulty by relating $N_X([0,T])$ to the count of zeroes 
within a suitable cover of $[0,T]$, for certain 
random analytic function $f:\S \to \C$ on a thin strip.
Thereby, complex analytic tools 
allow us to replace exponential moments of zero counts by   
more regular integrals of $\log |f(z)|$. After this reduction, the 
core challenge of our strategy remains in the need to sharply estimate
fractional moments of products of many dependent Gaussian variables. 
This highly non-trivial task (even for integer moments, see
\cite{SL} and the references within), requires    
our assumption \eqref{eq:reg-spect-dens}, in order 
to get suitable diagonally dominant covariance matrices.

\medskip
The paper is organized as follows.
In Section \ref{sec:sign} we prove Proposition \ref{t:lower},
Theorem \ref{t:sc} and Proposition \ref{p:suff-A}. 
The remainder of the paper is devoted to the proof of Theorem \ref{t:upper}.
In Section \ref{sec:main} this theorem is reduced to the key 
Proposition \ref{p:fracmoment}, concerning fractional moments of products 
of $\C$-valued Gaussian random variables. Proposition \ref{p:fracmoment} 
is proved in Section \ref{sec:mom}, building on the auxiliary 
results about weakly-correlated Gaussian 
variables that we 
establish in Section \ref{sec:cond}.


\section{Proofs of Proposition \ref{t:lower}, Theorem \ref{t:sc} and Proposition \ref{p:suff-A}}
\label{sec:sign}

\subsection{Proof of Proposition \ref{t:lower}}
Assume \abbr{wlog}
that $c(\eta) \leq 1$ and $C \ge e$. It suffices to consider 
$T \ge 1$. Fixing small $\delta >0$, by 
the mean-value theorem we have that 
$N_{X}([0,T]) \ge N^\star_Y([T/\delta]-1)$ for the 
stationary centered Gaussian sequence 
$Y_k:=\delta^{-1} \int_0^\delta X(\delta k + t) dt$. 
It is further easy to check that 
\begin{equation}\label{def:rY}
\gamma_k := \E Y_0 Y_k =  \delta^{-2} \, \int_0^\delta \int_0^\delta 
r (\delta k + t - s) ds \, dt  \,,
\end{equation}
corresponds to the spectral density 
\[
p_Y(\lm) = \frac{1}{\delta} 
\sum_{n \in \Z} p \Big(\frac{\lm+2\pi n}{\delta}\Big) 
\sinc^2 \Big(\frac{\lm}{2} + \pi n\Big)\,
\quad \lm\in[-\pi,\pi],
\]
where $\sinc(\lm) := \frac{\sin \lm}{\lm}$. Note that for $p(\cdot)$ bounded and continuous, 
$p_Y$ is also continuous (by dominated convergence). 
Here $r(0)=1$, $r'(0)=0$ and $-r''(t)$, 
being the characteristic function of the finite measure $\lm^2 p(\lm) d\lm$,
is continuous at $t \to 0$. We thereby get from \eqref{def:rY} 
that $\gamma_0 \to 1$ and 
$2 \delta^{-2} (\gamma_1-\gamma_0) \to r''(0)$ when  
$\delta \downarrow 0$. By a Gaussian computation 
$\P(Y_k Y_{k+1} < 0) = \frac{1}{\pi} \arccos(\gamma_1/\gamma_0)$ 
and it follows that  
\[
\inf_{T \ge 1} \frac{1}{T} \E N_Y^\star([T/\delta]-1) \ge (\delta^{-1}-2)
\P(Y_0 Y_1 < 0) \to \alpha \,.
\]
As a result of the preceding, 
we get \eqref{eq:lower-tail} by considering 
\eqref{conc:sign-changes} for $\delta=\delta(\eta)>0$ small enough.
\qed

\subsection{Proof of Theorem \ref{t:sc}}
We shall use the following easy consequence of weak convergence.
\begin{lemma}\label{l:continuity}
Let $(Y_0,Y_1)$ be a zero-mean jointly Gaussian, having  
$\E [Y_0^2] > 0$, $\E [Y_1^2]>0$ 
and $\alpha_\xi := \P(\, Y_0 Y_1 < \xi)$.
If   
the 
covariance matrices $\bSig^{(m)}$ of the 
zero-mean Gaussian vectors $(W^{(m)}_0,W^{(m)}_1)$ converge  
to the covariance matrix $\bSig$ of $(Y_0,Y_1)$, then 
\begin{equation}\label{eq:cont-m-g}
\lim_{\xi \to 0} \lim_{m \to \infty} \alpha^{(m)}_{\xi} = \alpha_0 \,,
\qquad \qquad 
\alpha^{(m)}_\xi := \P (\, W_0^{(m)} W_1^{(m)} < \xi \,).
\end{equation}
\end{lemma}
\begin{proof} Since $Y_0$ and $Y_1$ have positive variances, the
\abbr{cdf} of $Y_0 Y_1$ is continuous, so 
the weak convergence of $(W^{(m)}_0,W^{(m)}_1)$ to $(Y_0,Y_1)$ 
implies that for any $\xi$ fixed, 
$\alpha^{(m)}_\xi \to \alpha_\xi$ as $m \to \infty$.
Further, the monotone function $\xi \mapsto \alpha_\xi$ is then 
continuous and \eqref{eq:cont-m-g} holds.
\end{proof}

\smallskip
We turn to the proof of Theorem \ref{t:sc}. \abbr{wlog} normalize 
to have $\E Y_0^2 =1$. Let $\eta>0$ be given. 
For any $m \ge 2$ we approximate $Y$ by an $(m-1)$-dependent process using 
the following construction (see e.g.\ \cite[Proof of Theorem A]{BD}).
Let $\{a_k : k \in \Z \}$ denote the Fourier coefficients of the continuous function 
$\sqrt{p_Y(\lm)}$ on $[-\pi,\pi]$, and define 
$a_k^{(m)} := \left(1-\frac{|k|}{m-1}\right)_+ a_k$.
Then $Y_k = W^{(m)}_k + Z^{(m)}_k$, where $\{W^{(m)}_k : k \in \Z \}$ is an
$(m-1)$-dependent, 
centered, stationary Gaussian sequence with 
covariance $\E [W^{(m)}_0 W^{(m)}_n] = \sum_k a_k^{(m)} a^{(m)}_{k+n}$,
while by Fejer's theorem,
the spectral density $p^{(m)}_{Z}(\lm)=\left(\sqrt{p_Y(\lm)}  -\sqrt{p^{(m)}_{W}(\lm)}\right)^2$
of the centered, stationary Gaussian  
sequence $\{Z^{(m)}_k : k \in \Z \}$ converges to zero as $m \to \infty$,
uniformly on $[-\pi,\pi]$. Namely, 
$\vep_m := \sup_\lm \{ p^{(m)}_Z(\lm) \} \to 0$ as $m \to \infty$.

By stationarity, $\E N_Y^\star (T) = \alpha_0 T$ for
$\alpha_0 := 
\P(Y_0 Y_1 < 0)$.
Our assumption that the spectral measure $\rho_Y$ has a 
continuous density implies that $|r_Y(1)|<1$, hence
the covariance matrix $\bSig$ of $(Y_0,Y_1)$ is positive-definite.
Further, by construction, the covariance matrices $\bSig^{(m)}$ 
of $(W_0^{(m)},W_1^{(m)})$ converge to $\bSig$ when $m \to \infty$
and Lemma~\ref{l:continuity} applies.
In particular, there exist $\xi \in (0,1]$ and 
$m_\star < \infty$ so $\alpha^{(m)}_{3\xi} \le \alpha_0 + \eta$
whenever $m \ge m_\star$. Further, for such $\xi$, 
$m \ge m_\star$ and any $R := \xi/\delta \ge 1$, 
\[
\{Y_k Y_{k+1}<0\} \subseteq \{W^{(m)}_k W^{(m)}_{k+1} <3\xi \} 
\cup \{ |W^{(m)}_k| \ge R \}
\cup \{ |W^{(m)}_{k+1}| \ge R \}
\cup \{ |Z^{(m)}_k| \ge \delta \}
\cup \{ |Z^{(m)}_{k+1}| \ge \delta \}\,.
\]
Thus, the following bound applies for the upper tail of $N_Y^\star(T)$,
\[
\P\big(N_Y^\star(T) - \alpha_0 T \ge 8 \eta T \big)
\le \P\big( N_{m,\xi}(T) - \alpha^{(m)}_{3\xi} T \ge \eta T \big) 
+ 2\P\big(N_{m}^R(T) \ge 2\eta T\big)
+ 2 \P\big(N_{Z^{(m)}} (T) \ge \eta T\big) \,,
\]
where 
\begin{align}\label{def:Nm-xi}
N_{m,\xi} (T)  := \sum_{k=0}^{T-1} \mathbf{1}_{\{W^{(m)}_k W^{(m)}_{k+1} < 3 \xi \}},\quad
N_m^R (T) := \sum_{k=0}^{T-1} \mathbf{1}_{\{|W^{(m)}_k | \ge R \}}, \quad
N_Z (T)  := \sum_{k=0}^{T-1} \mathbf{1}_{\{|Z_k| \ge\delta \}} \,.
\end{align}
The zero-mean, 
$[-1,1]$-valued variables 
$I_k := {\bf 1}_{\{W^{(m)}_{k} W^{(m)}_{k+1} < 3 \xi\}} - \alpha^{(m)}_{3\xi}$ 
are $m$-dependent. Hence, setting $n_T:=\lfloor T/m \rfloor$ we get
by stationarity, followed by Hoeffding's inequality 
for the i.i.d. variables $\{I_{j m}\}_j$, that for $m \ge m_\star$
\begin{align}\label{eq:bd-Nm-xi}
\P\Big( N_{m,\xi} (T) \ge (\alpha^{(m)}_{3\xi} + \eta) T \,\Big)  
 \le m \max_{n \in \{n_T,n_T+1\}} 
 \P\Big( \sum_{j=0}^{n-1} I_{j m} \ge \eta n \Big) \le  
  m e^{-n_T \eta^2/2} \,.
\end{align}
Since $\E [(W^{(m)}_0)^2] \le 1$, fixing $R<\infty$ with 
$\P(|Y_0| \ge R) \le \eta$, we have that
$\widehat{\alpha}^{(m)}_R := \P(|W^{(m)}_0| \ge R) \le \eta$ for all $m \ge 1$.
Hence, by stationarity and the $m$-dependence of the $[-1,1]$-valued 
zero-mean $\widehat{I}_k := {\bf 1}_{\{|W^{(m)}_{k}| > R \}} - \widehat{\alpha}^{(m)}_R$,
applying once more Hoeffding's inequality, we get that 
\begin{align}
\P\Big(N_m^R(T) \ge 2 \eta T \Big) &\le 
\P\Big( \sum_{k=0}^{T-1} \widehat{I}_k \ge \eta T\Big) 
\le m \max_{n \in \{n_T,n_T+1\}} 
\P \Big( \sum_{j=0}^{n-1} \widehat{I}_{j m} \ge \eta n \Big) \le 
m e^{-n_T \eta^2/2} \,.
\label{eq:bd-Nm-R}
\end{align}
Finally, with $\E [ (Z^{(m)}_0)^2 ] \le 2\pi \vep_m$,  
from Markov's inequelity and \cite[identity (7)]{BD} at  
$\theta_m = \vep_m^{-1/2}$, we deduce that for all $T$ large enough 
\begin{equation}\label{eq:bd-Nm-Z}
\P\Big(N_{Z^{(m)}} (T) \ge \eta T \Big) \le 
e^{-\theta_m \delta \eta T} \E \Big[ e^{\theta_m
\sum_{k=0}^{T-1} |Z^{(m)}_k|} \Big]
\le e^{-(\theta_m \delta \eta -27) T} \,. 
\end{equation}
To complete the proof of the upper tail, combine
\eqref{eq:bd-Nm-xi}--\eqref{eq:bd-Nm-Z} taking $m \ge m_\star$ so 
large that $\theta_m \delta \eta \ge 28$.

Turning to prove the lower tail, set $\xi \in (0,1]$ and
$m_\star<\infty$ so $\alpha^{(m)}_{-3\xi} \ge \alpha_0-\eta$ 
whenever $m \ge m_\star$ and note that for any $R=\xi/\delta \ge 1$,   
\[
\{Y_k Y_{k+1} \ge 0\} \subseteq \{W^{(m)}_k W^{(m)}_{k+1} \ge - 3\xi\} 
\cup \{ |W^{(m)}_k| \ge R \}
\cup \{ |W^{(m)}_{k+1}| \ge R \}
\cup \{ |Z^{(m)}_k| \ge \delta \}
\cup \{ |Z^{(m)}_{k+1}| \ge \delta \} \,.
\]
Thus, recalling  from \eqref{def:sign-changes} and \eqref{def:Nm-xi} that 
\[  
N_Y^\star(T) = T -
\sum_{k=0}^{T-1} {\bf 1}_{\{Y_k Y_{k+1} \ge 0\}}  \,,
\qquad 
N_{m,-\xi} (T) = T -
\sum_{k=0}^{T-1} {\bf 1}_{\{W^{(m)}_k W^{(m)}_{k+1} \ge -3 \xi \}} \,,
\]
we have for any $\eta>0$, the bound  
\[
\P\big(\alpha_0 T - N_Y^\star(T) \ge 8 \eta T \big)
\le \P\big( \alpha^{(m)}_{-3\xi} T - N_{m,-\xi} (T) \ge \eta T \,\big) 
+ 2 \P\big(N_m^R (T) \ge 2 \eta T\big)
+ 2 \P\big(N_{Z^{(m)}}(T) \ge \eta T\big) \,.
\]
We have already established exponentially small in $T$ upper bounds on the two left-most terms
(in \eqref{eq:bd-Nm-R} and \eqref{eq:bd-Nm-Z}),  and  
re-running the derivation of \eqref{eq:bd-Nm-xi} for 
$I_k := \alpha_{-3\xi}^{(m)} - {\bf 1}_{\{W_k^{(m)} W_{k+1}^{(m)} < - 3\xi\}}$
yields such a bound on $\P( 
\alpha^{(m)}_{-3\xi} T - N_{m,-\xi} (T) \ge \eta T)$.
\qed

\subsection{Proof of Proposition \ref{p:suff-A}}
(a). Recall that $\int |r(t)|dt < \infty$ for $r(\cdot)$ of 
\eqref{eq:cov-spec} implies that $\rho$ has a 
continuous, bounded density $p(\lm)$. 
Assuming $\rho$ (hence $p(\lm)$) is supported 
on $[-K,K]$, fix
a compactly supported even function $\psi(\cdot)$  
such that $\psi(\lm) \equiv 1$ on $[-K,K]$
and $\sum_{\ell=0}^{2} \|\psi^{(\ell)}\|_\infty \le 1$.
Setting $h_{\ell,y}(\lm):=[\lm \varphi(\lm y)]^\ell \psi(\lm)$ 
and $r_{\ell;\psi}(x;y)$ via \eqref{dfn:ry-ell} but 
with $\psi(\cdot)$ replacing $p(\cdot)$, we find that
\begin{equation}\label{alt:r1-psi}
r'_{1;\psi}(x;y) = \int \sin(\lm x) h_{1,y}(\lm) d\lm 
= \frac{1}{x^2} \int \sin(\lm x) h''_{1,y}(\lm) d\lm 
\end{equation}
(getting the \abbr{rhs} for $x \ne 0$ upon twice integrating by parts).
One easily verifies that
\begin{equation}\label{eq:c-psi}
c_\psi := 2 \max_{\ell=0,1,2} \sup_{\ell |y| < 2 \kappa'} 
\{ \|h''_{\ell,y}\|_1 + \|h_{\ell,y}\|_1 \}  < \infty
\end{equation}
hence $|r'_{1;\psi}(x;2y)| \le c_\psi/(1+x^2)$ for any 
$|y| < \kappa'$ and all $x \in \R$. Since $\psi(\lm) p(\lm) = p(\lm)$, it 
follows that for  $g(x) := c_\psi \int dt |r(t)|/[1+(x-t)^2]$, 
any $|y| < \kappa'$ and $x \in \R$,
\begin{equation}\label{eq:bd-r1y}
|r'_1(x;2y)| = \big|\int r'_{1;\psi}(x-t;2y) r(t) dt\big| \le g(x) \,.   
\end{equation}
Likewise, 
\begin{equation}\label{alt:r2-psi}
- r''_{2;\psi} (x;y) = \int \cos(\lm x) h_{2,y} (\lm) d\lm
= \frac{1}{x^2} \int \cos(\lm x) h''_{2,y} (\lm) d\lm \,,
\end{equation}
hence 
$|r''_{2;\psi}(x;y)| \le c_\psi/(1+x^2)$ for all $|y| < \kappa'$,
yielding similarly to \eqref{eq:bd-r1y} that 
\begin{equation}\label{eq:bd-r2y}
|r''_2(x;y)| =  \big|\int r''_{2;\psi}(x-t;y) r(t) dt \big| \le g(x) \,.
\end{equation}
The same argument shows that 
\begin{equation}\label{eq:bd-r0y}
|r(x)| = \big| \int r_{0;\psi}(x-t) r(t) dt \big| \le g(x) \,.
\end{equation}
Taking $x_\star=1$ we thus find, in view of 
\eqref{eq:bd-r1y} and \eqref{eq:bd-r2y}, that
$\omega_\star(k) \le 3 \sum_{j \ge k} g(j)$
and \eqref{eq:reg-spect-dens} follows from the finiteness of 
\[
\sum_{j=1}^\infty g(j) = c_\psi \int dt |r(t)| \Big[ 
\sum_{j=1}^\infty \frac{1}{1+(j-t)^2} \Big] \le 2 c_\psi 
\Big[ \sum_{j \ge 0} \frac{1}{1+j^2} \Big] \int |r(t)| dt \,.
\]
(b). If $\rho(\cdot)$ of unbounded support satisfies \eqref{eq:exp-mom-nd},
then the covariance $r(\cdot)$ of \eqref{eq:cov-spec} is 
real-analytic and a.s. the sample path $t \mapsto X(t)$ is in $C^\infty(\R)$.
Suppose also that
for some $\kappa_o \in (0,\kappa/2)$ the covariance $r(\cdot;\kappa_o)$ 
of \eqref{dfn:r-kappa-o} is integrable. The latter implies that the 
measure $\cosh(2\kappa_o \lm) d\rho(\lm)$
has a continuous, bounded density $p_{\kappa_o}(\lm)$, hence
$\rho(\cdot)$ has the continuous, bounded density 
$p(\lm)=\psi(\lm) p_{\kappa_o}(\lm)$
for the even, $(0,1]$-valued, integrable function 
$\psi(\lm):=1/\cosh(2\kappa_o \lm)$.  
It is easy to verify that \eqref{eq:c-psi} remains valid for 
such choice of $\psi(\lm)$, provided $\kappa'<\kappa_o$. Further, 
in this case $|h_{\ell,y}(\lm)| \to 0$ and $|h'_{\ell,y}(\lm)| \to 0$ as
$|\lm| \to \infty$, whenever $\ell |y| < 2 \kappa'$, justifying the 
integration by parts that lead to the right-most equality in both
\eqref{alt:r1-psi} and  \eqref{alt:r2-psi}.  
The convolution identities
\eqref{eq:bd-r1y}, \eqref{eq:bd-r2y} and \eqref{eq:bd-r0y} apply
upon replacing $r(t)$ by $r(t;\kappa_o)$,
as do the corresponding bounds, albeit with 
$g(x):=c_\psi \int dt |r(t;\kappa_o)|/[1+(x-t)^2]$,
so the integrability of $|r(\cdot;\kappa_o)|$ indeed 
suffices for \eqref{eq:reg-spect-dens}.
\qed

\section{Analytic Extension, Jensen's formula and de-correlation}\label{sec:main}

\subsection{An analytic extension and its properties}
Under \eqref{eq:exp-mom-nd} the covariance kernel
$r:\R\to\R$ of the process $X(t)$ analytically extends 
to the strip $\S_{\kappa}=\{z \in \mathbb{C} : \ |\text{Im}(z)|<\kappa\}$, 
by plugging $t=z$ in \eqref{eq:cov-spec}. Utilizing this fact, we next 
construct a complex analytic, mean zero, 
Gaussian function $f : \S = \S_{\kappa/2} \to \C$ which is at the center
of our proof of Theorem \ref{t:upper}.
\begin{proposition}
\label{t:ext} 
For a real, stationary mean-zero, Gaussian process $X$ that 
satisfies \eqref{eq:exp-mom-nd}  
there exist a complex analytic, zero mean, Gaussian 
$f: \S := \S_{\kappa/2} \to \C$ such that:
\newline
(a) The function $f(\cdot)$ is conjugation equivariant, namely $f(\bar{z})=\overline{f(z)}$.
\newline
(b) The covariances of $f(\cdot)$ are given by the formulae
\begin{equation}
\label{p:corrf}
K(z,w) := \E [f(z)\overline{f(w)}]= r(z-\bar{w}); \qquad 
\E [f(z)f(w)]=r(z-w) \qquad \forall z,w \in \S \,.
\end{equation}
(c) The law of $z \mapsto f(z)$ is stationary under real translations
and $\{f(t+i0)\}_{t\in \R} \stackrel{d}{=} \{X(t)\}_{t\in \R}$.
\end{proposition}
\begin{proof} 
As $\rho$ an even real-valued measure,
there exists an orthonormal basis
(\abbr{ONB}) for $\mathcal{L}^2_\rho(\R)$ 
composed of Hermitian functions $\{\p_n\}$ 
(ie. with $\p_n(-\lm)=\overline{\p_n(\lm)}$). 
For such a basis and $e_z(\lambda) := e^{i \lm \bar z}$, $z \in \S_{\kappa/2}$  let 
\begin{equation}\label{eq:def-psi-n}
\psi_n(z) := \langle \p_n , e_z \rangle_{\mathcal{L}^2_\rho(\R)}
 = \int_\R \p_n(\lm) \overline{e_z(\lm)} d \rho(\lm) = \int_\R e^{-i \lm z} \p_n(\lm) d \rho(\lm) \,,
\end{equation}
and for i.i.d.\ coefficients $\zeta_n \sim \calN_\R(0,1)$ consider the 
random series
\[
f(z) := \sum_n \zeta_n \psi_n(z) \,.
\]
Having \eqref{eq:exp-mom-nd} hold for $\kappa$,  
standard arguments (see \cite[Chapter 3, Thm. 2]{Kah} or \cite[Lemma 2.2.3]{GAFbook}) yield that the series defining $f(\cdot)$ converges 
almost surely to a zero-mean, complex analytic Gaussian function
on $\S=\S_{\kappa/2}$, having there the covariance
\[
K(z,w)= \E [f(z)\overline{f(w)}] = \sum_n \psi_n(z)\overline{\psi_n(w)} \,.
\]
Since $\{\p_n\}$ are Hermitian and $\rho$ is even and real-valued,
it follows that $\overline{\psi_n(z)} = \psi_n(\bar z)$ and we get part (a) upon taking the conjugate in the defining series for $f(z)$.
Further, since $\{\p_n\}$ is an \abbr{ONB} 
in $\mathcal{L}^2_\rho(\R)$  
\[
K(z,w) =\sum_n\langle\p_n,e_z\rangle_{\mathcal{L}^2_\rho(\R)} 
\langle e_w,\p_n \rangle_{\mathcal{L}^2_\rho(\R)}=
\langle e_w,e_z \rangle_{\mathcal{L}^2_\rho(\R)} 
= r(z-\bar w)\,,
\]
as stated in \eqref{p:corrf} (and the \abbr{rhs} of 
\eqref{p:corrf} then follows from part (a)). 
The formulas \eqref{p:corrf} are invariant 
to real shifts $(z,w) \mapsto (z + t,w+t)$, $t \in \R$ hence the 
Gaussian function $f(\cdot)$ is stationary with respect to such real shifts.
To complete the proof of part (c), note that by part (a) 
the function $f(z)$ is real-valued when $z\in \R$ and the covariance kernel 
of \eqref{p:corrf} coincides for $z,w \in \R$ with the original covariance
$r:\R\to\R$ of the given real Gaussian process $X$.
\end{proof}
\begin{remark} 
{\rm Recall that
$\Re f(z) = [f(z) +\overline{f(z)}]/2$, $\Im f(z) = [f(z) - \overline{f(z)}]/(2i)$ with~\eqref{p:corrf} determining the covariance between the 
real and imaginary parts of $f(z)$ and $f(w)$, $z,w \in \S$. 
By Proposition \ref{t:ext}(c), when $\Im(z)=\Im(w)$ 
the latter depend only on $w-z$ so \abbr{wlog} we may set
$\Re(z)=0$. Specifically, for $|y| < \kappa/2$ and $x \in \R$ 
we have
\begin{align}\label{formula1}
\E [ \Re (f(iy)) \Re (f(x+iy)) ] &=\frac{1}{2} [\Re (r(x+2iy))+r(x)] 
= \int_{\R} \cos(\lm x) \cosh^2 (\lm y) d\rho(\lm)
\,, \\
\label{formula2}
\E [ \Im (f(iy)) \Im (f(x+iy)) ] &=\frac{1}{2} [\Re (r(x+2iy))-r(x)] 
= y^2 \int_{\R} \cos(\lm x) \lm^2 \varphi^2 (\lm y) d\rho(\lm)
\,, \\
\E [ \Re (f(iy)) \Im (f(x+iy)) ] &= \frac{1}{2}\Im (r(x+2iy)) 
\qquad \quad \; = - y \int_{\R} \sin(\lm x) \lm \varphi(2 \lm y)  d\rho(\lm) \,,
\label{formula3}
\end{align}
where $\varphi(\lm):=\sinh(\lm)/\lm$ 
and the 
\abbr{rhs} of \eqref{formula1}--\eqref{formula3} follows 
from \eqref{eq:cov-spec} and the even symmetry of the 
spectral measure $\rho(\cdot)$.}
\end{remark} 

We next utilize \eqref{formula1}--\eqref{formula3} to deduce from Assumption A
the absolute summability of the corresponding correlations,
uniformly in $\S_{\kappa'}$.
\begin{lemma}
\label{l:decay}
For $f(\cdot)$ of Proposition \ref{t:ext}, 
consider the vector $\bro (z) \in [-1,1]^4$ 
of correlations 
between $[\Re(f(z)), \Im(f(z))]$ and $[\Re(f(iy)),\Im(f(iy))]$ 
when $y=\Im(z)$. Then, for 
$x_\star$ and $0<\kappa'<\kappa/2$ of Assumption A, 
\begin{equation}
\label{eq:summable-corr}
\omega(k) := 4 \sup_{|y| < \kappa'} \big\{ 
\sum_{j=k}^{\infty} \|\bro (j x_\star + i y)\| \big\} \to 0 \quad
\mbox{ for } \;\; k \to \infty 
\end{equation}
and in particular 
\begin{equation}\label{eq:summable-cov}
\sup_{|y| < \kappa'} \Big\{  \sum_{j \in \Z} |r(j x_\star + 2iy)| \Big\} 
< \infty \,.
\end{equation}
\end{lemma}
\begin{proof} In view of \eqref{formula1} and \eqref{formula2}, for any 
$|y| < \kappa/2$ and $x \in \R$,
\begin{align}\label{dfn:vI}
\vI (y) := y^{-2} \var(\Im(f(x+iy)) & = \int_{\R} \lm^2 
\varphi^2(\lm y) d\rho(\lm) \,, \\
\vR (y) :=  \var(\Re(f(x+iy))  & = \int_{\R} \cosh^2(\lm y) d\rho(\lm) \,,
\label{dfn:vR}
\end{align}
are uniformly in $y$, bounded away from zero  
and thanks to \eqref{eq:exp-mom-nd}, 
\[
\var(\Im(f(x+iy))  
\le \var(\Re(f(x+iy))  
\le r(2iy)
\]
is uniformly bounded over $|y| \le \kappa' < \kappa/2$.
Further, by parts (b) and (c) of Proposition \ref{t:ext},
\[
|r(x + 2iy)|  = |\E [ \overline{f(iy)} f(x + iy) ]| 
\le r(2iy)\, \| \bro (|x|+iy) \|_1 \,.
\]
Thus, if \eqref{eq:summable-corr} holds, then necessarily 
$\omega(1) 
$
is finite and 
\eqref{eq:summable-cov} must hold as well.
Turning to show \eqref{eq:summable-corr}, we have from 
\eqref{formula1}--\eqref{formula3}, that the coordinates of the vector 
$\bro(x+iy)$ satisfy the relations
\begin{align}\label{eq:II-dom-RR}
\bro_{\textsf{RR}}(x+iy) =  (1-\beta_y) \, \bro_{\textsf{II}}(x+iy) 
+ \beta_y r(x) 
\,, \qquad 
\bro_{\textsf{RI}}(x+iy) &= - \bro_{\textsf{IR}}(x+iy) \,,\\
\vI (y)\,
\bro_{\textsf{II}}(x+iy) 
= - r''_2(x;y) \,, 
\qquad \sqrt{\vR (y) \vI (y)} \,
\bro_{\textsf{IR}}(x+iy) 
& = r_1'(x;2y) \,,
\label{eq:rho-RI}
\end{align}
for $\beta_y := 2r(0)/(r(2iy)+r(0)) \in (0,1)$ and the 
$x$-derivatives of
$r_{\ell}(\cdot)$ of \eqref{dfn:ry-ell}. Recalling that
$\inf_y \{ \vI (y) \wedge \vR (y) \} \ge c^{-1}$
we deduce from \eqref{eq:II-dom-RR}--\eqref{eq:rho-RI} that 
\[
\|\bro (j x_\star + i y)\| \le r(j x_\star) 
+ 2 c |r''_2(j x_\star;y)| + 2 c |r_1'(j x_\star;2y)| \,,
\]
hence $\omega(k) \le \textcolor{red}{8} (c \vee 1) \omega_\star(k)$ and \eqref{eq:summable-corr} 
follows from our assumption \eqref{eq:reg-spect-dens}.
\end{proof}

\subsection{Relating real and complex zeroes}
Thanks to the second part of Proposition \ref{t:ext}(c), 
for any $\D \subseteq \S$ containing $[0,T]$ we have
\begin{equation}\label{eq:anal-ubd} 
N_X([0,T]) \le N_f(\D) = | \{ z \in \D : f(z)=0 \} | \,. 
\end{equation}
For $\kappa'$ 
of Assumption A and $\delta \in (0,\kappa'/2)$,
let $\B_j(r)$ denote the ball of radius $r$ centered at 
$x_j:= (2 j - 1) \delta$. We shall use the 
bound \eqref{eq:anal-ubd} with the disjoint union of 
$n := \lceil T/(2\delta) \rceil$ balls
\[
\D = \D_{n,\delta} := \bigcup_{j=1}^{n} \B_j(\delta) \,,
\] 
further estimating the value of $N_f(\B_j(\delta))$ 
by Jensen's 
enumeration formula for the zeroes of a complex analytic function (see \cite[Section 5.3.1]{Ahlfors}). 
Specifically, for $\beta \in [0,\log 2]$ define the integral
\begin{equation}\label{dfn:Ij}
\Gamma_j(\beta) :=  \int_{-1/2}^{1/2} 
\log |f(x_j+\delta e^{\beta} e^{i 2\pi \theta})| d \theta \,.
\end{equation}
With such choices $\delta e^\beta < \kappa' < \kappa/2$, so 
$\B_j(\delta e^\beta) \subset \S$ and
Jensen's formula tells us that for each~$j$ 
\begin{equation}\label{eq:jensen}
 \int_{\delta}^{\delta e^{\beta}} N_f(\B_j(r)) \frac{\mathrm{d}r}{r} 
 =  \Gamma_j(\beta) - \Gamma_j(0).
\end{equation}
Since $r \mapsto N_f(\B_j(r))$ is non-decreasing, 
from \eqref{eq:jensen} we 
deduce that 
\begin{equation}\label{eq:bds}
N_f( \B_j(\delta) ) \leq \frac{1}{\beta} \big[ \Gamma_j(\beta) - \Gamma_j(0) \big] \leq N_f(  \B_j (\delta e^{\beta})  ).
\end{equation}
Sum \eqref{eq:bds} over $j$ to get
\begin{equation}\label{eq:basic-bd}
N_f(\D_{n,\delta})  \le \frac{1}{\beta} \sum_{j=1}^{n} \widehat{\Gamma}_j(\beta)   
- \frac{1}{\beta} \sum_{j=1}^{n} \widehat{\Gamma}_j(0) + 
\sum_{j=1}^{n} \E [ N_f(\B_j  (\delta e^{\beta})  ) ]  \,,
\end{equation}
where $\widehat{\Gamma}_j(\cdot) := \Gamma_j(\cdot) - \E \Gamma_j(\cdot)$.
The next lemma shows that for small positive $\delta$ and $\beta$ 
the right-most 
(non-random)
term in \eqref{eq:basic-bd} is at most 
$(\alpha + \eta)T$.

\begin{lemma}\label{l:approx}
Suppose that \eqref{eq:exp-mom-nd} holds and the spectral measure 
$\rho(\cdot)$ is non-atomic. There exist  
$\delta_\star(\eta)$ and $\beta_\star(\eta)$ positive, such that
for any $\delta \le \delta_\star(\eta)$, 
$\beta \le \beta_\star(\eta)$ 
and all $T \ge 1$,  
\begin{equation}\label{eq:mean-Nf-balls}
\frac{1}{T} \sum_{j=1}^n \E [N_f( \B_j(\delta e^{\beta})  )] \leq 
\alpha + \eta \,.
\end{equation}
\end{lemma}
\begin{proof} 
Since the Gaussian function $f(z)$ has non-atomic spectral measure,
\[
\E [ N_f ([0,1] \times [-r,r]) ] = \alpha + \mu_f([-r,r]) \,, 
\qquad \forall r \ge 0 \,,
\]
where $\mu_f(\cdot)$ is some absolutely continuous, non-negative 
measure on $\R$
(see \cite[Theorem 1]{Nmean}). Further, $z \mapsto f(z)$ is stationary
under real translations,
hence for any 
$x_j \in \R$ and $r \in [\delta e^\beta,\tfrac 1 2] \cap \Q$
\[
\E [N_f( \B_j(\delta e^{\beta}) )] \le \E[ N_f([-r,r]^2) ] 
= 2r (\alpha + \mu_f([-r,r])) \,.
\]
With $n \le \frac{T}{2\delta} +1$, 
the \abbr{lhs} of \eqref{eq:mean-Nf-balls} is thus for $T \ge 1$ and
$\delta < \tfrac{1}{4}$, at most
\begin{equation*}
h(\delta,\beta) := 
(1 + 2 \delta) e^\beta \big(\alpha + \mu_f([-\delta e^\beta,\delta e^\beta])
\big) 
\end{equation*}
and we are done, since $h(\cdot,\cdot)$ is continuous with $h(0,0)=\alpha$.
\end{proof}

\subsection{Reducing Theorem \ref{t:upper} to the decorrelation of moments.}
Fixing $\eta>0$, in view of \eqref{eq:anal-ubd} 
and \eqref{eq:basic-bd} it suffices 
for \eqref{eq:upper-tail} to show that for $\beta=\beta^\star$
and $\delta^\star$ as in Lemma \ref{l:approx}, there exist 
$\delta \in (0,\delta^\star]$
and $c=c(\eta,\beta,\delta)>0$ so that for all $n$ large enough
\begin{equation}\label{eq:fluct}
\P\Big ( \sum_{j=1}^n \widehat{\Gamma}_j(\beta) \geq \eta \beta \delta n \Big) 
+ \P\Big ( \sum_{j=1}^n \widehat{\Gamma}_j(0)  \leq -\eta \beta \delta n \Big) 
\leq e^{-c n} \,.
\end{equation}
To this end, let $x_\star$ be as in Assumption A and 
consider $\delta \in (0,\delta^\star]$ 
such that $x_\star/(2 \delta) := \ell_\star \in \N$.
Then, to utilize the decay of correlations in Lemma 
\ref{l:decay}, fix $\ell = k \ell_\star$ for some $k \in \N$ and 
cover $\{1,\ldots,n\}$ by the disjoint union of $\ell$ sets 
$\cS_\tau :=  \{\ell - \tau, 2\ell -\tau,\ldots,m\ell-\tau\}$
(namely $\tau=0,\ldots,\ell-1$), with $m=\lceil n/\ell \rceil \ge 2$. 
By stationarity
of $f(\cdot)$ under real translations,  
the law of 
$\sum_{j \in \cS_\tau} \widehat{\Gamma}_j (\cdot)$
is independent of $\tau$. 
Setting $\xi:=\eta \beta/10$,
by a union bound on
the $\ell$ choices of $\tau$, it suffices for
\eqref{eq:fluct} to show that 
some $c=c(\xi,\delta,\ell)>0$ and all $m$ large enough
\begin{equation}\label{eq:fluct-k}
\P\Big ( \sum_{j \in \cS_0} \widehat{\Gamma}_j(\beta) 
\geq 5 \xi \delta m \Big) 
+ \P\Big ( \sum_{j \in \cS_0} \widehat{\Gamma}_j(0)  \leq -5 \xi \delta m \Big) 
\leq  2e^{-2 c \ell m} \,.
\end{equation}
A standard application of the exponential Markov inequality reduces
this task 
for $c = \vep \xi \delta/(2 \ell)$,
into showing that for some $\vep=\vep(\xi,\delta,\ell)>0$ 
and all large enough $m$,  
\begin{equation}
\E \Big[ \exp (\vep \sum_{j=1}^m \widehat{\Gamma}_{j\ell}(\beta))
\Big]  \leq e^{4 \vep \xi \delta m}  \qquad \& \qquad
\E \Big[ \exp (-\vep \sum_{j=1}^m \widehat{\Gamma}_{j\ell}(0))
\Big]  \leq e^{4 \vep \xi \delta m} \,.
\label{eq:moment-ul}
\end{equation}
Upon setting $z_{\beta}(\theta) :=\delta e^{\beta} e^{i \pi \theta} - \delta$,
we get from \eqref{dfn:Ij} that 
\[
\sum_{j=1}^m \widehat{\Gamma}_{j\ell}(\beta) =  
\frac{1}{2} \int_{-1}^{1} S_m(z_\beta(\theta)) d \theta 
\]
where, $2 \delta \ell 
= k x_\star$ thanks to our choice of $\ell$, so 
\begin{equation}\label{def:Am}
S_m(z) := 
\sum_{j=1}^m 
\Big\{ \log |f(j k x_\star+z)| - \E [\log |f(z)|]
\Big\} \,.
\end{equation}
Thus, applying Jensen's inequality for the convex functions 
$\exp(\pm \vep \, \cdot)$, 
further reduces the task of proving \eqref{eq:moment-ul} into 
showing that 
\begin{equation}\label{l:pointwise}
\sup_{|\theta| \le 1} \E \Big[ e^{\vep S_m(z_\beta(\theta))} \Big] 
\le e^{4\vep \xi \delta m} \qquad \& \qquad 
\sup_{|\theta| \le 1} \E \Big[ e^{-\vep S_m(z_0(\theta))} \Big] 
\le e^{4\vep \xi \delta m} \,.
\end{equation}
In view of the stationarity of $f(\cdot)$ under real translations, the 
law of $S_m(z)$ of \eqref{def:Am}
depends only on $\Im(z)$, hence in \eqref{l:pointwise}
we can \abbr{wlog} re-set 
$z_\beta(\theta) =iy$ 
for $y=\sin(\pi \theta) \delta e^{\beta}$.
Doing so, we consider for $|y| \le 2 \delta$,
the mean-zero, Gaussian variables
\begin{equation}\label{def:Gj}
G_j(y) := f(j x_\star + iy) \,, 
\end{equation}
and first relate $\E [\log |G_0 (y)|]$ which is part of $S_m(iy)$ to 
$\E [|G_0(y)|^{\pm \vep}]$.
\begin{lemma}
\label{l:reversejensen}
Given $\zeta>0$, for any 
$\vep \le \vep_0(\zeta)$ positive and 
all $|y| \le \kappa'$,
\begin{align}\label{eq:bd-mean-log}
\E [ |G_0(y)|^{\vep} ] & \leq (1+\vep \zeta)
\exp\Big( \vep \E \big[\log |G_0(y)|\big] \Big)\,, \\
\E [ |G_0(y)|^{-\vep} ] & \leq (1+\vep \zeta) 
\exp\Big( - \vep \E \big[\log |G_0(y)|\big] \Big)\,.
\label{eq:meanreal}
\end{align}
\end{lemma}
\begin{proof} We re-write \eqref{eq:bd-mean-log}--\eqref{eq:meanreal}
in terms 
of $L(y):=\log |G_0(y)| - \E [\log |G_0(y)|]$
and the non-negative function 
$g_\vep(x):=|\vep|^{-1} (e^{\vep x} - \vep x -1)$, as 
$\E [g_{\pm \vep} (L(y))] \le  \zeta$ and 
prove the lemma by showing that 
\begin{equation}\label{eq:bd-mean-log-alt}
\lim_{\vep \downarrow 0} \,
\sup_{|y| \le \kappa'} \big\{ \, \E [g_{\pm \vep} (L(y))] \, \big\} = 0 \,.
\end{equation}
Since
$|g_{\pm \vep} (x)|  
\le \eta^{-1} e^{\eta |x|} :=\widetilde{g}_\eta (x)$ whenever $|\vep| \le \eta$
and 
$g_{\pm \vep} (\cdot) \to 0$ uniformly on compact subsets of $\R$, 
the uniform in $y$
convergence \eqref{eq:bd-mean-log-alt}, is a
consequence of having for some $\eta>0$,  
\begin{equation}\label{eq:unif-int-log}
\sup_{|y| \le \kappa'} \big\{ \, \E [ \widetilde{g}_\eta(L(y)){ \bf 1}_{\{|L(y)| \ge b\}}] \, \big\} 
\to 0 \quad \hbox { for } \quad b \to \infty \,.
\end{equation}
Further, $\widetilde{g}_\eta(\cdot)$ diverges at infinity, so 
\eqref{eq:unif-int-log} follows from having 
$\sup_{|y| \le \kappa'} \big\{ \E [\widetilde{g}_\eta^2(L(y))] \big\}$ finite, 
for which it suffices to verify that $\sup_{|y| \le \kappa'} 
\{ \E [|G_0(y)|^{\pm 2\eta}] \}$ 
is finite. 
For the latter, recall \eqref{p:corrf} that
$\E [|G_0(y)|^2] = r(2iy)$,
which for $|y| \le \kappa'$ is uniformly bounded above 
(as $\kappa'<\kappa/2$),  whereas
$\E [|G_0(y)|^{-1/2}] \le C r(0)^{-1/4}$ for some universal 
$C<\infty$,
since $|G_0(y)|^{-1/2} \le |X|^{-1/2}$ for the zero-mean 
$\R$-valued Gaussian $X=\Re(G_0(y))$ of $\var(X) = \vR(y) \ge r(0)$
(see \eqref{dfn:vR}).
\end{proof}

The next proposition, which is our main technical statement, 
bounds small positive and negative fractional moments of the product of 
our Gaussian variables from \eqref{def:Gj}, after a suitable dilution. 
\begin{proposition}
\label{p:fracmoment}
For any $\zeta>0$ there is $\vep_\star(\zeta)>0$  
such that for $\vep \le \vep_\star$,
$k \ge k_\star(\zeta,\vep) \in \N$,
large enough $m$, and all $|y| < \kappa'$,
\begin{align}\label{eq:fracmomentu}
M_m(\vep) &:= \E \Big[ \prod_{j=1}^m |G_{jk}(y)|^{\vep} \Big] 
\leq e^{2 \vep \zeta m} \E [ |G_{0}(y)|^{2\vep} ]^{m/2} \,, \\
M_m(-\vep) &:= \E \Big[ \prod_{j=1}^m |G_{jk}(y)|^{-\vep} \Big] 
\leq e^{2 \vep \zeta m} \E [ |G_{0}(y)|^{-2\vep} ]^{m/2} \,.
\label{eq:fracmomentl}
\end{align}
\end{proposition}

We proceed to obtain \eqref{l:pointwise} from Proposition
\ref{p:fracmoment}.
In view of \eqref{def:Am} and \eqref{def:Gj}, the \abbr{lhs} 
of \eqref{l:pointwise} amounts (after setting $z_\beta=iy$), to 
\begin{equation}\label{eq:pointwise-ubd}
\E \Big[ \prod_{j=1}^m |G_{jk}(y)|^\vep \Big] \le e^{4 |\vep| \xi \delta m}
\exp \Big( \vep m \E \log |G_0(y)| \Big)  \,,
\qquad \forall |y| \le 2 \delta  \,.
\end{equation}
Proceeding to show \eqref{eq:pointwise-ubd}, 
we set $\zeta:= \xi \delta>0$ and a positive 
$\vep \le \vep_\star (\zeta) \wedge \vep_0(\zeta)/2$, 
so Lemma \ref{l:reversejensen} applies at $2\vep$, then
fix $k \ge k_\star(\zeta,\vep)$ large enough 
as needed for Proposition \ref{p:fracmoment}.
Combining now the bound \eqref{eq:fracmomentu} with 
\eqref{eq:bd-mean-log} at $2\vep$ and 
the elementary inequality $(1+2\vep \zeta) \le e^{2 \vep \zeta}$,
yields the bound \eqref{eq:pointwise-ubd}. 
Similarly, the \abbr{rhs} of \eqref{l:pointwise} amounts to 
the inequality \eqref{eq:pointwise-ubd} at $-\vep<0$, so 
having the control of \eqref{eq:meanreal} on the $-\vep$-moment
of $G_0(y)$ in terms of
$\E \log|G_0(y)|$, we deduce 
that the \abbr{rhs} of \eqref{l:pointwise} 
follows from the bound \eqref{eq:fracmomentl}.

In conclusion, we have by now reduced the proof of Theorem \ref{t:upper} 
to the de-correlated moment computations of Proposition \ref{p:fracmoment},
to which we devote Sections \ref{sec:cond} and \ref{sec:mom}.

\section{Diagonally dominant Gaussian laws}\label{sec:cond}

We establish here a few preparatory results 
about weakly correlated, centered, $\C$-valued Gaussian vectors.
Our results are phrased in terms of 
\begin{equation}\label{eq:XYG}
f(j x_\star + i y) := G_j(y) := \sqrt{\vR (y)} X_{j}(y) 
+ i |y| \sqrt{\vI (y)} Y_j(y) \,,
\end{equation}
for 
$\vI(y)$ and $\vR(y)$ of \eqref{dfn:vI}--\eqref{dfn:vR},
standard, $\R$-valued Gaussian $X_j(y)$, $Y_j(y)$ 
which are independent of each other
(see \eqref{formula3}),
and all absolute constants are independent of $y \in (-\kappa',\kappa')$.
Such results apply whenever $\E [|G_j|^2]$ are uniformly 
bounded above and below, provided the covariance matrix 
of the Gaussian $\{G_j\}$ is diagonally dominant,
in the sense that 
the correlations between $\{X_j,Y_j\}$ 
and $\{X_{j+k},Y_{j+k}\}$ are absolutely summable (in $k$),
with a uniform (in $j$), tail decay, as in \eqref{eq:summable-corr}.

Our first result (needed for proving \eqref{eq:fracmomentu}),
is a uniform a-priori control on the second moment of the
product of such Gaussian variables (assuming only that they
have summable covariances, as in \eqref{eq:summable-cov}).
\begin{lemma}
\label{l:bad}
There exists finite $C_\star \ge 1$ 
such that for all $|y| < \kappa'$
and any finite $J \subset \N$,
\begin{equation}\label{eq:2nd-mom-bd} 
\E \Big[ \prod_{j\in J} |G_{j}(y)|^{2} \Big] \leq C_\star^{2 |J| } \,.
\end{equation}
\end{lemma}

\begin{proof} For centered Gaussian $(Z_1,\ldots,Z_n) \in \C^n$ with 
$r(\ell,\ell') = \E [Z_\ell \overline Z_{\ell'}]$ 
one has that  
\begin{align}\label{eq:univ-2nd-mom-bd}
M_n &:= \E[\, \prod_{\ell=1}^n |Z_{\ell}|^2 \,] 
\le \prod_{\ell=1}^n R_\ell \quad \hbox{ where } \quad
R_\ell := \sum_{\ell'=1}^n |r(\ell,\ell')|  \,.
\end{align}
Indeed, by Wick's formula (see \cite[Lemma 2.1.7]{GAFbook}),    
\begin{equation}\label{eq:complex-wick}
M_n = \sum_{\pi} \prod_{j=1}^n r (j,\pi(j)) \,,
\end{equation}
where we sum over permutations $\pi$ of 
$\cS := \{1,2,\ldots,n\}$. To bound
$M_n$ from above, replace each term $r(j,\pi(j))$ by
$|r(j,\pi(j))|$, whereupon having only non-negative terms, 
further bound $M_n$ by summing over the larger collection
of all functions $\pi:\cS \mapsto \cS$. The latter sum
is precisely the product of 
$\{ R_\ell : \ell \in \cs \}$, yielding \eqref{eq:univ-2nd-mom-bd}.
 
Now apply \eqref{eq:univ-2nd-mom-bd} for the centered complex 
Gaussian $\{ G_j(y) , j \in J \}$ and bound $R_\ell$ by summing 
over all $\ell' \in \Z$.
Setting
$C_{y} := \sum_{j \in \Z} |r(j x_\star + 2iy)|$, we thus get
from Prop. \ref{p:corrf}(b) that for any $J$ and $|y| < \kappa'$, 
\[
\sup_{\ell} \{ R_\ell \} \le C_y \le \sup_{|y| < \kappa'} \{C_{y} \} := C_\star^2 
\]
which is finite by Lemma \ref{l:decay} (see \eqref{eq:summable-cov}).
\end{proof}

Let $\cJ_k$ denote the collection of all 
finite 
sets $\{j_1,j_2,\ldots,j_n\} \subset \N$,
where $j_i \ge j_{i-1} + k$ for $j_0:=0$ and any $i \in [1,n]$. 
Note that for the sequence 
$\omega(k) \to 0$ of \eqref{eq:summable-corr}, 
and $J \in \cJ_k$, the 
centered, $\R$-valued, Gaussian vector 
$\mathbf{Z}=(X_0(y),Y_0(y), \{X_j(y),Y_j(y)\}_{j \in J})$ 
has covariance matrix $\bSig := \bI - \bS$ such that 
for all $|y| < \kappa'$,
\begin{equation}\label{dfn:diag-dom}
\| \bS \|_{\infty \to \infty} := \sup_{{\bf x} \ne {\bf 0}} \Big\{ 
\frac{\| \bS {\bf x} \|_\infty}{\| {\bf x}\|_\infty} \Big\} =
\max_{j} \Big\{ \sum_{j'} |\bS_{jj'}| \Big\} \le \omega(k) \,.
\end{equation}
We next detail three elementary 
properties of Gaussian vectors having such a diagonally 
dominant covariance matrix.
\begin{lemma}
\label{l:diagdominant}
Suppose $\mathbf{Z}=({\bf Z}_1,{\bf Z}_2)$ is
centered, $n$-dimensional $\R$-valued Gaussian vector 
and $\cov({\bf Z}) := \bI - \bS$ with 
$\|\bS\|_{\infty \to \infty} \le \omega < 1$. Then, setting $\widehat{\omega}_i := \omega^i/(1-\omega)$,
$i=0,1,2$, we have that:

\noindent
(a) All entries of the \abbr{psd} matrix $\cov({\bf Z}_1) - \cov({\bf Z}_1|
{\bf Z}_2)$ 
are within $[-\widehat{\omega}_2,\widehat{\omega}_2]$.
\newline
(b) The inequality  
$\|\E [ {\bf Z}_1 \, | {\bf Z}_2 ] \|_\infty \le \widehat{\omega}_1 \| {\bf Z_2} \|_\infty $ holds.
\newline
(c) The density $f_{\bf Z} (\cdot)$ of $\mathbf{Z}$ with 
respect to i.i.d. standard Normal variables, is such that 
\begin{equation}\label{eq:bd-rnd}
f_{\bf Z} ({\bf z}) \le
(\widehat{\omega}_0)^{n/2} \exp(\widehat{\omega}_1 \|{\bf z}\|_2^2/2) \,.
\end{equation}
\end{lemma}
\begin{proof} (a). Our assumption that $\|\bS\| \le \omega < 1$, implies that  
$\bSig^{-1} = \sum_{n \ge 0} \bS^n$ satisfies
\begin{equation}\label{eq:bd-inv}
\| \bI - \bSig^{-1} \|_{\infty \to \infty} \le \sum_{n=1}^\infty \omega^n = \widehat{\omega}_1 \,, 
\quad 
\|\bSig^{-1}\|_{\infty \to \infty} 
\le \sum_{n=0}^\infty \omega^n  = \widehat{\omega}_0 \,.
\end{equation}
With
$\bSig_{11} := \cov({\bf Z}_1)$,
$\bSig_{22} := \cov({\bf Z}_2)$,
$\bSig_{12} = (\bSig_{21})^\star = \cov({\bf Z}_1,{\bf Z}_2)$, 
and 
$\bSig_{1|2}:=\cov({\bf Z}_1 | {\bf Z}_2)$,
recall that (see \cite[Exer. 2.1.3]{GAFbook}),
\begin{equation}\label{dfn:cond-cov}
\bSig_{11} - \bSig_{1|2} =
\bSig_{12} \bSig_{22}^{-1} \bSig_{21} \,. 
\end{equation}
The $L_1$-norm of each column of 
$\bSig_{21}$ is by assumption at most $\omega$. Further, the \abbr{rhs}
of \eqref{eq:bd-inv} applies to $\bSig_{22}^{-1}$, which by
\eqref{dfn:cond-cov} implies that all entries of 
$\bSig_{11}-\bSig_{1|2}$ are indeed 
within $[-\widehat{\omega}_2, \widehat{\omega}_2]$.
\newline 
(b). Since
$\bmu := \E [ {\bf Z}_1 \, | {\bf Z}_2 ]=\bSig_{12}\bSig_{22}^{-1}{\bf Z}_2$
and the 
\abbr{rhs} of \eqref{eq:bd-inv} applies for $\bSig_{22}$,
we deduce as in part (a), that
necessarily $\|\bmu\|_\infty \le \omega \, \widehat{\omega}_0 
\|{\bf Z}_2\|_\infty$.
\newline
(c). The matrix norm of \eqref{dfn:diag-dom} dominates the spectral norm. 
In particular, from the \abbr{lhs} of \eqref{eq:bd-inv} we deduce that 
\[
\langle {\mathbf z}, (\bI - \bSig^{-1}){\mathbf z} \rangle 
\le \widehat{\omega}_1 \| {\mathbf z} \|_2^2 \,.
\]
Further, the \abbr{rhs} of \eqref{eq:bd-inv} implies
that all eigenvalues of $\bSig^{-1}$ are within $[0,\widehat{\omega}_0]$, 
hence the density 
\[
f_{\bf Z}({\bf z}) = |\bSig^{-1}|^{1/2} \exp\big(\frac{1}{2}
\langle {\mathbf z}, (\bI - \bSig^{-1}){\mathbf z} \rangle \big) \,,
\]
satisfies the bound \eqref{eq:bd-rnd}, as claimed.
\end{proof}

Relying on diagonal dominance to lower bound the conditional 
variances, as in Lemma \ref{l:diagdominant}(a), we get the following
negative moment bound 
(which will later be useful when proving \eqref{eq:fracmomentl}).
\begin{lemma}
\label{l:worstcase}
For some finite $k_o,C_o \ge 1$ and $\vep_o>0$, all $\vep \leq \vep_o$, $|y| < \kappa'$ and any $J \in \cJ_{k_o}$  
\begin{equation}\label{e:worstcase}
\E\Big[ |G_0(y)|^{-4\vep} \; \mid \; \{ G_{j}(y) : j\in J\} \big] \leq 
C_o  \,.
\end{equation}
\end{lemma}
\begin{proof} Since $|z| \ge |\Re(z)|$ it suffices to show
that \eqref{e:worstcase} holds when $\Re(G_0(y))$ replaces 
$|G_0(y)|$. Further, we only need to do so for say 
$\vep_o = 1/8$, as it thereafter extends by Jensen's 
inequality (and the convexity of 
$g(x)=x^p$ on $\R_+$ when $p=\vep_o/\vep \ge 1$),
to all $\vep \le \vep_o$. To this end, recall that the 
conditional law of 
$\Re(G_0(y))$
given the finite $\C$-valued Gaussian collection 
$\{ G_{j}(y), j \in J \}$, is Gaussian of some 
non-random (conditional) variance $v=\vR(y;J)$
and random mean $\widehat{\mu} \sqrt{v}$ (see \cite[Exer. 2.1.3]{GAFbook}).
With $\phi(\cdot)$ 
denoting the standard normal density, we thus have by scaling,
that the conditional expectation of $|\Re(G_0(y))|^{-1/2}$ is at most
\[
v^{-1/4} \sup_{\widehat{\mu} \in \R} \,\Big\{ \int_\R (|x-\widehat{\mu}| \wedge 1)^{-1/2} \phi(x) dx \Big\} := v^{-1/4} C_1 \,,
\]
for some finite constant $C_1 \le 1 + \phi(0) \int_{-1}^{1} |x|^{-1/2} dx$.
With $\vR(y)$ uniformly bounded below and 
$\omega(k_o) \le 1/3$ for some $k_o$ finite
(see \eqref{eq:summable-corr}), it follows that 
\[
u(k_o):=
(1 - \omega(k_o)) \inf_{|y| < \kappa'} \{ \vR(y) \} > 0 
\]
and we get \eqref{e:worstcase} with $C_o = u(k_o)^{-1/4} C_1$, upon showing that 
\begin{equation}\label{eq:bdd-below-var}
\inf_{J \in \cJ_{k_o},  |y| < \kappa'} \{\vR (y;J)\} \ge u(k_o) \,.
\end{equation}  
To this end, as $\E [X_0(y)^2]=1$ and $\omega(k_o) \le 1/2$,
from Lemma \ref{l:diagdominant}(a) we then have that 
\[
\frac{\vR(y;J)}{\vR(y)} 
= \E [ X_0(y)^2 \mid  \{X_{j}(y),Y_j(y): j \in J\} ] \ge \E [ X_0(y)^2 ] 
- \frac{\omega(k_o)^2}{1-\omega(k_o)} \ge 1 - \omega(k_o) \,,
\]
and \eqref{eq:bdd-below-var} follows.
\end{proof}

Next, for fixed $k \ge k_o$, $y$, $m$ and any threshold $\Delta>0$, we
define the collection  
\begin{equation}\label{dfn:bad}
B_\Delta := \{1 \le j \le m : |X_{jk}(y)| \vee |Y_{jk}(y)| > \Delta  \} \,,
\end{equation}
of ``bad" indices, and use diagonal dominance (specifically, Lemma 
\ref{l:diagdominant}(c)), to show that for large~$\Delta$ it is 
exponentially highly unlikely to have many bad indices.
\begin{lemma}
\label{l:badprob}
There exists $c(\Delta) \to \infty$ as $\Delta \to \infty$ such that 
for any $k \ge k_o$, all $|y|<\kappa'$ and non-random 
$B \subset \{ 1, \ldots, m\}$,
\begin{equation}\label{eq:large-badset}
\P(B \subseteq B_\Delta) \leq e^{-4 c(\Delta)|B|} \,.
\end{equation}
\end{lemma}
\begin{proof} Since the event $\{B \subseteq B_\Delta\}$ is the union of
\[
\{ |X_{jk}(y)| > \Delta, \; \forall j \in J \} \cap 
\{ |Y_{jk}(y)| > \Delta, \; \forall j \in B \setminus J \} \,, 
\] 
over the $2^{|B|}$ possible $J \subseteq B$, by Cauchy-Schwartz  
it suffices for \eqref{eq:large-badset} to show that 
for some $b(\Delta) := 8 c(\Delta) + \log 4 \to \infty$ as $\Delta \to \infty$,
both 
\begin{align}\label{eq:X-bad}
p_{J} (\Delta) &:= \P( |X_{jk}(y)| > \Delta, \; \forall j \in J) 
\le e^{-b(\Delta) |J|} \,, \\
q_{J} (\Delta) &:= \P( \, |Y_{jk}(y)| > \Delta, \; \forall j \in J)  \le 
e^{-b(\Delta) |J|} \,.
\label{eq:Y-bad}
\end{align}
To this end, recall that 
$1 - \omega(k) \ge 2/3 \ge 2 \omega(k)$, whenever $k \ge k_o$. 
Hence, from Lemma \ref{l:diagdominant}(c) we have the bound 
\[
p_{J} (\Delta)^{1/|J|} \le (1-\omega(k))^{-1/2} 
\E \Big[e^{\frac{\omega(k) X_0^2}{2(1-\omega(k))}} {\bf 1}_{\{|X_0| > \Delta\}} \Big]  
\le  \frac{2 \sqrt{3/2}}{\sqrt{2\pi}} \int_{\Delta}^\infty e^{-x^2/4} dx 
: = e^{-b(\Delta)}
\]
for which \eqref{eq:X-bad} holds. 
Exactly the same argument applies for $q_J(\Delta)$, 
yielding the bound \eqref{eq:large-badset}.
\end{proof}

We conclude the section by showing that,
thanks to Lemma \ref{l:diagdominant}(b), 
for large $k=k(\Delta,\vep)$ and any $J \in \cJ_k$,
the conditional expectation of $|G_0(y)|^{\pm \vep}$
given a realization of $\{X_{j}(y), Y_{j}(y) : j \in J \}$, 
all of whom are in a specified range $[-\Delta,\Delta]$, 
is within error $(1+o(\vep))$ of the 
unconditional expectation. 
\begin{lemma}
\label{l:goodmoment}
Let $H_J(y) := \max_{j \in J} \{ |X_{j}(y)|, |Y_j(y)| \}$.
There exist
$k_\star(\Delta,\zeta,\vep) : \R_+^3 \to [k_o,\infty)$
and $\vep_\star > 0$,
such that for
any $\Delta,\zeta>0$, $\vep \leq \vep_\star$, 
$J \in \cJ_{k_\star}$ and
$|y| < \kappa'$
\begin{align}\label{eq:goodmomentu}
 \E \Big[ |G_0 (y)|^{\vep} 
 \mid \, \{ G_{j}(y)\}_{j \in J} 
\Big] {\bf 1}_{\{H_J(y) \le \Delta\}}
&\leq e^{\vep \zeta}  \E [ |G_0(y)|^{\vep} ] \,, \\
\E \Big[ |G_0 (y)|^{-\vep} \mid \, \{G_{j}(y)\}_{j\in J} \Big] 
{\bf 1}_{\{H_J(y) \le \Delta\}}
& \leq e^{\vep \zeta}  \E [ |G_0(y)|^{-\vep} ]  \,.
\label{eq:goodmomentl}
\end{align} 
\end{lemma}
\begin{proof} We use the representation \eqref{eq:XYG}, dividing
\eqref{eq:goodmomentu} and \eqref{eq:goodmomentl} 
by $\vR(y)^{\pm \vep}$, respectively. Then, with 
$u := y^2 \vI(y)/\vR (y) \in [0,1]$
and $g_{u,\pm \vep}({\bf x}) := (x_1^2 + u x_2^2)^{\pm \vep/2}$, the
stated inequalities amount to
\begin{equation}\label{eq:equiv-goodmoment}
{\bf 1}_{\{H_J(y) \le \Delta\}}
\int_{\R^2} g_{u,\pm \vep}({\bf x}) f_{J}({\bf x}) d \gamma({\bf x}) 
\le e^{\vep \zeta} \int_{\R^2} g_{u,\pm \vep}({\bf x}) d \gamma ({\bf x}) \,,
\end{equation}
where $f_{J}(\cdot)$ is the Radon-Nikodym density of 
the conditional law of $(X_0(y),Y_0(y))$ with respect to the 
standard two-dimensional Gaussian law $\gamma$.
Recall Lemma \ref{l:diagdominant}(a)
that for any $J \in \cJ_k$, $k \ge k_o$, the 
two-dimensional covariance matrix 
$\bSig_{1|2} := \bI_2 - \bS$ of $(X_0(y),Y_0(y))$ given 
$\{G_j(y), j \in J\}$, satisfies
$\|\bS\|_{\infty \to \infty} \le \omega(k)^2/(1-\omega(k)) \le \omega(k)$. 
Further, by Lemma \ref{l:diagdominant}(b), 
the conditional mean $\bmu$ of $(X_0(y),Y_0(y))$ must satisfy
$\|\bmu\|_\infty \le 2 \omega(k) H_J(y)$. Here 
$\omega = \omega(k) \le 1/3$, so similarly 
to the derivation of \eqref{eq:bd-rnd}, we have for 
the (random)
two-dimensional Radon-Nikodym density $f_J(\cdot)$ that  
\begin{equation}\label{eq:bd-rd2}
f_{J}({\bf x}) = 
|\bSig_{1|2}|^{-1/2} \exp\Big \{ 
\frac{1}{2} \big( \langle {\bf x},{\bf x} \rangle -
\langle {\bf x} - \bmu , \bSig_{1|2}^{-1} ({\bf x} - \bmu) \rangle \big) \Big\}  
\le \widehat{f}_{\omega(k),H_J(y)} ({\bf x}) \,,
\end{equation}
where for any fixed $\Delta<\infty$,
\[
\widehat{f}_{\omega,\Delta} ({\bf x}) :=
(1-\omega)^{-1} \exp\Big\{
\omega (x_1^2+x_2^2 + 3 \Delta |x_1|+ 3 \Delta |x_2|) \Big\} \searrow 1\,,
\quad \mbox{ when } \quad \omega \searrow 0 \,.  
\]
Note that $g_{u,\vep} \le g_{1,\vep}$ and $g_{u,-\vep} \le g_{0,-\vep}$.
Further, both 
$g_{1,\vep} \cdot (1+\widehat{f}_{\omega,\Delta})$ and 
$g_{0,-\vep}\cdot (1+\widehat{f}_{\omega,\Delta})$ are in 
$\mathcal{L}^1_{\gamma}(\R^2)$
as soon as $\vep \le \vep_\star <1$ and $\omega < 1/2$. Consequently, per 
$\vep \le \vep_\star$ and $\Delta<\infty$, the functions
${\bf x} \mapsto g_{u,\pm \vep}({\bf x}) |\widehat{f}_{\omega,\Delta}({\bf x}) -1|$ 
are uniformly in $u$ (and $\omega \le 1/3$), integrable with respect 
to 
$\gamma$, and converge pointwise to zero as $\omega \searrow 0$. Thus, 
\[
\lim_{\omega \searrow 0} \sup_{u \in [0,1]} 
\Big| 
\int_{\R^2} g_{u,\pm \vep}({\bf x}) \widehat{f}_{\omega,\Delta}({\bf x}) d\gamma({\bf x})
-
\int_{\R^2} g_{u,\pm \vep}({\bf x}) d\gamma({\bf x})
\Big| = 0 \,,
\]
which together with \eqref{eq:bd-rd2} and \eqref{eq:summable-corr} 
imply the existence of finite $k_\star(\Delta,\zeta,\vep) \ge k_o$
such that \eqref{eq:equiv-goodmoment} holds whenever
$J \in \cJ_{k_\star}$
and $|y|<\kappa'$.
\end{proof}

\section{Moment computations: Proof of Proposition \ref{p:fracmoment}}
\label{sec:mom}

\subsection{Proof of \eqref{eq:fracmomentu}}
Since $c(\Delta)$ of Lemma \ref{l:badprob} is unbounded,
for any $\zeta>0$ and $\vep \le \vep_\star < 1$ 
(where $\vep_\star$ is from Lemma \ref{l:goodmoment}), we can take
$\Delta=\Delta(\zeta,\vep)$ so large that 
\begin{equation}\label{eq:choice-Delta}
C_\star e^{-c(\Delta)} \le \vep \zeta e^{\vep \zeta}
\E\big[|G_0(y)|^{2\vep}\big]^{1/2} \,,
\end{equation}
where $C_\star$ is the finite constant from Lemma \ref{l:bad}.
Given such $\Delta$, let
$h_{\vep,\Delta}(G) := |G|^\vep {\bf 1}_{\{
|X| \vee |Y| \le \Delta\}}$ (for $G(y)$, $X(y)$, $Y(y)$ 
related as in \eqref{eq:XYG}).
Then,
 fix $k \ge k_\star$ (also from
Lemma \ref{l:goodmoment}),
and partition the expression $M_m(\vep)$ of \eqref{eq:fracmomentu}
according to 
$B_\Delta$ of \eqref{dfn:bad}, to get that 
\begin{eqnarray}
\label{e:bad-part}
M_m(\vep) = \sum_{B} \E \Big[ \prod_{j \in B} |G_{jk}(y)|^{\vep} 
\, {\bf 1}_{\{B \subseteq B_\Delta\}} 
\prod_{j\in B^c} h_{\vep,\Delta} (G_{jk}(y))
\Big] \,,
\end{eqnarray}
where 
the sum is over all $B \subseteq \{1,\ldots,m\}$. For
$(1-\vep)/2 \ge 1/4$, using H\"older's inequality 
we bound the generic 
summand on the \abbr{rhs} of \eqref{e:bad-part} by
\begin{equation}
\label{e:cs2}
\E \Big [ \prod_{j \in B} |G_{jk}(y)|^2 \Big]^{\vep/2} \,
\P(B \subseteq B_\Delta)^{1/4} \,
\E \Big[ 
\prod_{j\in B^c} h_{2\vep,\Delta}(G_{jk}(y)) \Big]^{1/2} \,.
\end{equation}
Enumerating $B^c=\{j_1,j_2,\ldots\}$ with $1 \le j_1<j_2<\cdots \le m$
and utilizing the stationarity of $\{ G_{j}(y) \}_j$, we 
appeal sequentially for $s=1,\ldots$, to \eqref{eq:goodmomentu} 
with $k^{-1} J = B^c \setminus \{j_1,\ldots,j_s\}$ shifted backward by $j_s$, 
to deduce that (since $k \ge k_\star$),
\begin{equation}\label{e:good1}
\E \Big[ 
\prod_{j\in B^c} h_{2\vep,\Delta}(G_{jk}(y)) \Big] \le 
\Big( e^{2 \vep \zeta} \E\big[|G_0(y)|^{2\vep}\big] \Big)^{|B^c|} \,.
\end{equation}
Further bounding the left term of \eqref{e:cs2} via Lemma \ref{l:bad}
and the middle one via Lemma \ref{l:badprob},
we complete the proof by deducing
from \eqref{e:bad-part} and \eqref{e:good1} that 
\begin{align}
M_m(\vep) &\leq \sum_B C_\star^{|B| \vep} e^{-c(\Delta)|B|} 
\Big( e^{2 \vep \zeta} \E\big[|G_0(y)|^{2\vep}\big] \Big)^{|B^c|/2} 
\nonumber \\
& = \Big\{ C_\star^{\vep} e^{-c(\Delta)} + e^{\vep \zeta}
\E\big[|G_0(y)|^{2\vep}\big]^{1/2} \Big\}^m \le 
 e^{2 \vep \zeta m} \E\big[|G_0(y)|^{2\vep}\big]^{m/2} 
\label{e:goodbad}
\end{align}
(with the last inequality holding thanks to having chosen
$\Delta$ that satisfies \eqref{eq:choice-Delta}).
\qed

\subsection{Proof of \eqref{eq:fracmomentl}}
Here Lemma \ref{l:worstcase} replaces
Lemma \ref{l:bad}, so upon further reducing
$\vep$ to satisfy $\vep \le \vep_o$ we set 
$\Delta=\Delta(\zeta,\vep)$ so large that 
\begin{equation}\label{eq:choice-Delta-neg}
C_o e^{-c(\Delta)} \le \vep \zeta e^{\vep \zeta}
\E\big[|G_0(y)|^{-2\vep}\big]^{1/2} \,,
\end{equation}
where $C_o$ and $\vep_o$ are the finite constants
from Lemma \ref{l:worstcase}. Proceeding as
in the proof of \eqref{eq:fracmomentu}, for $k \ge k_\star \ge k_o$ 
we
partition the expression $M_m(-\vep)$ of \eqref{eq:fracmomentl}
according to $B_\Delta$ to get the identity \eqref{e:bad-part} 
at $-\vep$. Then, analogously to \eqref{e:cs2}, 
we apply H\"older's inequality 
to bound the generic summand on the \abbr{rhs} of that identity 
(now at $-\vep$), by 
\begin{equation}
\label{e:cs1}
\E \Big [ \prod_{j \in B} |G_{jk}(y)|^{-4 \vep} \Big]^{1/4} \,
\P(B \subseteq B_\Delta)^{1/4} \,
\E \Big[ 
\prod_{j\in B^c} h_{-2\vep,\Delta}(G_{jk}(y)) \Big]^{1/2} \,.
\end{equation}
The only difference \abbr{wrt} \eqref{e:cs2} is 
the first exponent $-4\vep$ instead of $-2$
(as $\E [|G_0(y)|^{-2}]=\infty$). The middle and 
last term of \eqref{e:cs1} are handled precisely 
as in the proof of \eqref{eq:fracmomentu}, upon 
appealing to Lemma \ref{l:badprob} and \eqref{eq:goodmomentl}, respectively.
With $\Delta$ satisfying \eqref{eq:choice-Delta-neg},
the proof of \eqref{eq:fracmomentl} is thus complete 
upon establishing that
\begin{equation}\label{e:good2}
\E \Big [ \prod_{j \in B} |G_{jk}(y)|^{-4 \vep} \Big] \le C_o^{|B|} \,.
\end{equation}
Similarly to the derivation of \eqref{e:good1}, upon enumerating 
$B=\{j_1 < j_2 < \ldots\}$ we get the bound \eqref{e:good2}
by repeated conditioning and using
\eqref{e:worstcase} for $s=1,2,\ldots$ with 
$k^{-1} J=B \setminus \{j_1,\ldots,j_s\}$ shifted backward by $j_s$.
\qed

\medskip
\noindent
{\bf Open problem:} 
Does the exponential upper tail of \eqref{eq:upper-tail} hold in case 
of covariance $r(t)=\sinc(t)$ 
(with spectral density $p(\lm)=\frac{1}{2} {\bf 1}_{[-1,1]}(\lm)$)?
Note that this covariance satisfies Assumption A (for $x_\star=2\pi$), 
apart from the lack of 
summability of the $r_1'(\cdot;2y)$ term in \eqref{eq:reg-spect-dens}.

\medskip
{\sc{Acknowledgments:}}
We thank Mikhail Sodin for motivating this research (in discussions with NF), and 
Manjunath Krishnapur for suggesting (to RB) that complex zeroes may be useful for showing concentration.

\bigskip

\bigskip

\footnotesize

  R.~Basu, \textsc{
International Centre for Theoretical Sciences, Tata Institute of Fundamental Research, Bangalore 560089, INDIA.
}\par\nopagebreak
  \textit{E-mail address}: \texttt{rbasu@icts.res.in}

  \smallskip

  A.~Dembo, \textsc{
Department of Mathematics, Stanford University, Stanford CA 94305, USA.
}\par\nopagebreak
  \textit{E-mail address}: \texttt{adembo@stanford.edu}

  \smallskip

  N.~Feldheim, \textsc{
Department of Mathematics,
Weizmann Institute of Science, Rehovot 76100, Israel. 
}\par\nopagebreak
  \textit{E-mail address}: \texttt{trinomi@gmail.com}

  \smallskip
  O. Zeitouni, \textsc{
Department of Mathematics,
Weizmann Institute of Science, Rehovot 76100, Israel. 
}\par\nopagebreak
  \textit{E-mail address}: \texttt{ofer.zeitouni@weizmann.ac.il}

\end{document}